\documentclass[11pt]{amsart}

\newtheorem{theorem}{Theorem}[section]

\newtheorem{proposition}[theorem]{Proposition}
\newtheorem{corollary}[theorem]{Corollary}

\theoremstyle{definition}

\newtheorem{example}[theorem]{Example}

\theoremstyle{remark}

\numberwithin{equation}{section}

\begin{document}

\title[Two-sided multiplication operators on the space of regular operators]
 {Two-sided multiplication operators on the space of regular operators}

\author[J. X. Chen]
{Jin Xi Chen }

\address{Department of Mathematics, Southwest Jiaotong
University, Chengdu 610031, PR China}
 \email{jinxichen@home.swjtu.edu.cn}

\author[A. R. Schep]
{Anton R. Schep}
\address{Department of Mathematics, University of South Carolina, Columbia, SC 29208}
\email{schep@math.sc.edu}

\thanks{ The first author  was  supported in part by China Scholarship Council (CSC) and was visiting the University of South Carolina when this work was completed.}

%    General info
\subjclass[2010]{Primary 46A40; Secondary 46B42, 47B65}
%\dedicatory{}

\keywords{regular operator, two-sided multiplication operator, Riesz space, Banach lattice}

\begin{abstract}
Let $W$, $X$, $Y$ and $Z$ be Dedekind complete Riesz spaces. For  $A\in L^{r}(Y, Z)$ and $B\in L^{r}(W, X)$ let $M_{A,\,B}$ be the two-sided multiplication operator from $L^{r}(X, Y)$ into $L^r(W,\,Z)$ defined by $M_{A,\,B}(T)=ATB$. We show that for every  $0\leq A_0\in L^{r}_{n}(Y, Z)$, $|M_{A_0, B}|(T)=M_{A_0, |B|}(T)$ holds for all $B\in L^{r}(W, X)$ and all $T\in L^{r}_{n}(X, Y)$. Furthermore, if $W$, $X$, $Y$ and $Z$ are Dedekind complete Banach lattices such that $X$ and $Y$  have  order continuous norms, then  $|M_{A,\, B}|=M_{|A|, \,|B|}$ for all $ A\in L^{r}(Y, Z)$   and all $B\in L^{r}(W, X)$. Our results generalize the related results of Synnatzschke and Wickstead, respectively.

\end{abstract}

\maketitle \baselineskip 4.5mm
\section{Introduction}
\par For an algebra $\mathcal{A}$ an operator of the form $T\mapsto\sum_{i=1}^{n}A_{i}TB_{i}$, where $A_i$, $B_i$ are fixed in $\mathcal{A}$, is referred to as an \textit{elementary operator} on $\mathcal{A}$. If $A,\,B\in\mathcal{A}$, we denote by $M_{A,\,B}$ the operator $T\mapsto ATB$. The operator $M_{A,\,B}$ is called  a \textit{basic elementary operator} or a \textit{two-sided multiplication operator}. The literature related to (basic) elementary operators is by now very large, much of it in the setting of $C^*$-algebras or in the Banach space setting. In this direction there are many excellent surveys and expositions. See, e.g., \cite{Ara, Curto, Mathieu, Saksman}.

\par For the study of two-sided multiplication operators in the setting of Riesz spaces (i.e., vector lattices) we would like to mention  the work of  Synnatzschke \cite{Synnatz}. The set of all regular operators (order continuous regular operators, resp.) from a Riesz space $X$ into a Dedekind complete Riesz space $Y$ will be denoted by $L^{r}(X, Y)$  ($L^{r}_n(X, Y)$, resp.). When $Y=\mathbb{R}$, we write $X^{\sim}$ and $X^{\sim}_{n}$ respectively instead of $L^{r}(X, \mathbb{R})$ and $L^{r}_n(X, \mathbb{R})$. They are likewise Dedekind complete Riesz spaces. Let $W$, $X$, $Y$ and $Z$ be Dedekind complete Riesz spaces. For all $A\in L^r(Y,\,Z)$ and $B\in L^r(W,\,X)$, $M_{A,\,B}:T\in L^r(X,\,Y)\mapsto ATB\in L^r(W,\,Z)$ is a regular operator,  and hence  the modulus $|M_{A,\,B}|$ of $M_{A,\,B}$ exists in $ L^{r}\big(L^{r}(X, Y), L^{r}(W, Z)\big)$. It is interesting to know about the relationship of $|M_{A,\,B}|$ with $|A|$ and $|B|$. Among other things, Synnatzschke \cite[Satz 3.1]{Synnatz} proved the following result:\\[6pt]
\indent a) If $0\leq B_0\in L^r(W,\,X)$, then $|M_{A,\,B_0}|=M_{|A|,\,B_0}$, $M_{A,\,B_0}\vee\, M_{C,\,B_0}=M_{A\vee C,\,B_0}$ hold for all $A,\,C\in L^r(Y,\,Z)$. \\[6pt]
\indent b) If $0\leq A_0\in L^{r}_n(Y,\,Z)$ and $Y^{\sim}_{n},\,Z^{\,\sim}_{n}$ are total, then  we have $|M_{A_0, B}|(T)=M_{A_0, |B|}(T)$ and $(M_{A_0,\,B}\vee M_{A_0,\,D}) (T)=M_{A_0,\,B\vee D}(T)$ for all $B,\,D\in L^{r}(W, X)$ and all $T\in L^{r}_{n}(X, Y)$.
\par \noindent Hereby  $Y^{\sim}_{n}$ is\textit{ total} if $Y^{\sim}_{n}$ separates the points of $Y$. Synnatzschke uesd (a) to establish (b) by taking adjoints of operators. For his purpose, the hypothesis of both $Y^{\sim}_{n}$ and $Z^{\,\sim}_{n}$ being total is essential.

\par Recently, Wickstead \cite{Wickstead} showed that if $E$ is an atomic Banach lattice with order continuous norm and $A,\,B\in L^r(E)$, then $|M_{A,\,B}|=M_{|A|,\,|B|}$ and $\|M_{A,\, B}\|_{r}=\|A\|_{r} \|B\|_r$. In his proofs he depended heavily upon the `atomic' condition.
\par In this note, we  generalize the related results of Synnatzschke and Wickstead, respectively.  We remove the condition of order continuous duals being total in \cite[Satz 3.1(b)]{Synnatz} and show that for every  $0\leq A_0\in L^{r}_{n}(Y, Z)$, $|M_{A_0, B}|(T)=M_{A_0, |B|}(T)$ holds for all $B\in L^{r}(W, X)$ and all $T\in L^{r}_{n}(X, Y)$. Furthermore, if $W$, $X$, $Y$ and $Z$ are Dedekind complete Banach lattices such that $X$ and $Y$  have  order continuous norms (not necessarily atomic), then  $|M_{A,\, B}|=M_{|A|, \,|B|}$ and $\|M_{A,\, B}\|_{r}=\|A\|_{r} \|B\|_r$ hold for all $ A\in L^{r}(Y, Z)$   and all $B\in L^{r}(W, X)$.

\par Our notions are standard. For the theory of Riesz spaces and regular operators, we refer the reader to the monographs \cite{AB, M, Za}.
 \section{The modulus of the two-sided multiplication operator}
   We start with two examples which serve to illustrate that the order continuous dual $X^{\sim}_{n}$ of a Dedekind complete Riesz space $X$ need not be total. This justifies our effort to generalize the result of Synnatzschke \cite[Satz 3.1 b)]{Synnatz}.
 \begin{example}(1) Let $(\Omega,\,\Sigma,\,\mu)$ be a nonatomic finite measure space. Then the Dedekind complete Riesz space $X=L_p(\Omega,\,\Sigma,\,\mu)$ ($0<p<1$) satisfies $X^{\sim}_{n}=X^{\sim}=\{0\}$. This result is due to M. M. Day. (cf. \cite[Theorem 5.24, p. 128]{AB2}).

 \par (2) Let $K$ be a compact Hausdorff space. It is well known that $C(K)$ is Dedekind complete if and only if $K$ is Stonian (i.e., extremally disconnected). A Hausdorff compact Stonian space $K$ such that $C(K)^{\sim}_{n}$ is total  is called hyper-Stonian. Dixmier gave a characterization of hyper-Stonian spaces: $K$ is hyper-Stonian if and only if $C(K)$ is isomorphic to a dual Banach lattice (cf. \cite[Theorem 2.1.7]{M}). He also gave a Dedekind complete $C(K)$-space which is not isomorphic to a dual space (see, e.g., \cite[p. 99, Problems 4.8 and 4.9]{Alb} for details). That is, such a $C(K)$ is  Dedekind complete, but $C(K)^{\sim}_{n}$ is not total.

 \end{example}
\begin{proposition}\label{Prop 2.1}
 Let $W$, $X$, $Y$ and $Z$ be  Riesz spaces with $X$, $Y$ and $Z$ Dedekind complete. Let $0\leq A_0\in L^{r}_{n}(Y, Z)$. Then we have $|M_{A_0, B}|(T)=M_{A_0, |B|}(T)$ and, equivalently, $M_{A_0,\,B}\vee M_{A_0,\,D} (T)=M_{A_0,\,B\vee D}(T)$ for all $B,\,D\in L^{r}(W, X)$ and all $T\in L^{r}_{n}(X, Y)$.
\end{proposition}
\begin{proof}
 For $B\in L^{r}(W, X)$ and $0\leq T\in L^{r}_n(X, Y)$, we  have to prove that $|M_{A_0, B}|(T)= M_{A_0, |B|}(T)$.  Clearly we have $|M_{A_0, B}|(T)\leq M_{A_0, |B|}(T)$, since $|M_{A_0, B}|\leq M_{A_0, |B|}$ holds in $L^{r}\big(L^{r}(X, Y), L^{r}(W, Z)\big)$. For the reverse inequality, let $w\in W^{+}$. By a formula for the modulus of regular operators in \cite[Theorem 1.21(3)]{AB} or \cite[Theorem 20.10(i)]{Za} we have$$\left(\sum_{i=1}^{n}|Bw_i|:n\in\mathbb{N},\, 0\leq w_{i}\in W,\, \sum_{i}w_i=w\right)\uparrow|B|w.$$Since $A_0$ and $T$ are both positive order continuous operators, $A_0T$ is likewise an  order continuous positive operator from $X$ into $Z$. It follows that
\begin{eqnarray*}
M_{A_{0},\,|B|}(T)(w)&=&A_{0}T|B|w\\
&=&\sup\left(\sum_{i=1}^{n}A_0T|Bw_i|:n\in\mathbb{N},\, 0\leq w_{i}\in W,\, \sum_{i}w_i=w\right)
\end{eqnarray*}
For each $1\leq i\leq n$, let $P_i$ be the order projection from $X$ onto the band generated by $(Bw_{i})^+$ in $X$ and let $Q_i=P_i-I$, where $I$ is the identity operator on $X$. Clearly, $$P_i\perp Q_i,\quad |P_i|+|Q_i|=I,\quad P_{i}Bw_i=(Bw_{i})^+,\quad (P_i+Q_i)Bw_i=|Bw_i|,$$and $$|TP_i|+|TQ_i|=T.$$Therefore, for each $i$ we have
\begin{eqnarray*}
A_0T|Bw_i|&=&(A_0TP_i+A_0TQ_i)Bw_i\\
&\leq&\big(|A_0(TP_i)B|+|A_0(TQ_i)B|\big)w_i\\
&\leq&\left(\sup\Bigg\{\sum_{j=1}^{m}|A_0T_jB|:m\in\mathbb{N},\, T_j\in L^{r}(X,Y),\, \sum_{j}|T_j|=T\Bigg\}\right)w_i\\
&=&\left(\sup\Bigg\{\sum_{j=1}^{m}|M_{A_{0},\,B}(T_j)|:m\in\mathbb{N},\, T_j\in L^{r}(X,Y),\, \sum_{j}|T_j|=T\Bigg\}\right)w_i\\
&=&|M_{A_{0},B}|(T)(w_i).
\end{eqnarray*}
Hence, from this it follows that
\begin{eqnarray*}
M_{A_{0},\,|B|}(T)(w)&=&\sup\left(\sum_{i=1}^{n}A_0T|Bw_i|:n\in\mathbb{N},\, 0\leq w_{i}\in W,\, \sum_{i}w_i=w\right)\\
&\leq&\sup\left(\sum_{i=1}^{n}|M_{A_{0},B}|(T)(w_i):n\in\mathbb{N},\, 0\leq w_{i}\in W,\, \sum_{i}w_i=w\right)\\
&=&|M_{A_{0},B}|(T)(w),
\end{eqnarray*}
which implies that $|M_{A_0, B}|(T)\leq M_{A_0, |B|}(T)$ for all $B\in L^{r}(W, X)$ and all $0\leq T\in L^{r}_{n}(X, Y)$.
\end{proof}
In general we can not expect that  $|M_{A_0, B}|=M_{A_0, |B|}$  holds for all $B\in L^{r}(W, X)$. That is, the linear operator $M_{A_0,\, \cdot}:B\in L^{r}(W, X)\rightarrow L^{r}\big(L^{r}(X, Y), L^{r}(W, Z)\big)$ is not necessarily a Riesz homomorphism. In the last section we give a counterexample to illustrate this. However, for Banach lattices with order continuous norms the situation is quite different. The next result is a consequence of the above proposition and the earlier result of Synnatzschke \cite[Satz 3.1]{Synnatz}, which generalizes Theorem 3.1 of Wickstead \cite{Wickstead} recently obtained for atomic Banach lattices with order continuous norms.
\begin{corollary}\label{Corollary 2.2}
Let $W$, $X$, $Y$ and $Z$ be Banach lattices such that $X$, $Y$  have  order continuous norms and $Z$ is Dedekind complete.  Then we have $|M_{A,\, B}|=M_{|A|, \,|B|}$ for all $ A\in L^{r}(Y, Z)$   and all $B\in L^{r}(W, X)$.
\end{corollary}
\begin{proof}
Let $\mathcal{M}:L^{r}(Y, Z)\times L^{r}(W, X)\rightarrow L^{r}\big(L^{r}(X, Y), L^{r}(W, Z)\big)$ be the bilinear operator defined via $\mathcal{M}(A, B)=M_{A,\,B}$. Clearly, $\mathcal{M}$ is positive.
Since $X$ and $Y$  are Banach lattices with  order continuous norms, we have $L^{r}_{n}(X,\, Y)=L^{r}(X,\, Y)$ and $L^{r}_{n}(Y,\, Z)=L^{r}(Y,\, Z)$. From  Proposition \ref{Prop 2.1} above and the result of Synnatzschke \cite[Satz 3.1]{Synnatz} it follows that for every $0\leq A_0\in L^{r}(Y, Z)$ and every $0\leq B_0\in L^{r}(W, X) $, $\mathcal{M}(A_0,\,\cdot)$ and $\mathcal{M}(\cdot\,,B_0)$ are both Riesz homomorphisms. Hence, for all $A\in L^{r}(Y, Z)$ and all $B\in L^{r}(W, X)$ we have
\begin{eqnarray*}
|M_{A, B}|&=&|\mathcal{M}(A, B)|\\&=&|\mathcal{M}(A^{+}-A^{-}, B^{+}-B^{-})|\\
&=&|\mathcal{M}(A^+,\,B^+)-\mathcal{M}(A^+,\,B^-)-\mathcal{M}(A^-,\,B^+)+\mathcal{M}(A^-,\,B^-)|\\
&=&\mathcal{M}(A^+,\,B^+)+\mathcal{M}(A^+,\,B^-)+\mathcal{M}(A^-,\,B^+)+\mathcal{M}(A^-,\,B^-)\\
&=&\mathcal{M}(|A|,\,|B|)=M_{|A|,\,|B|}.
\end{eqnarray*}
Here we are using the fact that the terms $\mathcal{M}(A^+,\,B^+)$, $\mathcal{M}(A^+,\,B^-)$, $\mathcal{M}(A^-,\,B^+)$ and $\mathcal{M}(A^-,\,B^-)$ are pairwise disjoint.
\end{proof}
Let $W$ and $X$ be Banach lattices with $X$ Dedekind complete. Recall that $L^{r}(W,\,X)$ is a Dedekind complete Banach lattice under the regular norm $\|B\|_r:=\||B|\|$ for every $B\in L^{r}(W,\,X)$. Note that $M_{A,\,B}$ is a regular operator from $L^{r}(X, Y)$ into $L^{r}(W, Z)$.
The following result deals with the regular norms of two-sided multiplication operators. Its proof is based on Corollary \ref{Corollary 2.2}
\begin{corollary}\label{Corollary 2.3}
If $W$, $X$, $Y$ and $Z$ be Banach lattices such that $X$, $Y$  have  order continuous norms and $Z$ is Dedekind complete, then $\|M_{A,\, B}\|_{r}=\|A\|_{r} \|B\|_r$ for all $ A\in L^{r}(Y, Z)$   and all $B\in L^{r}(W, X)$.
\end{corollary}

\begin{proof}
We first assume that $ 0\leq A\in L^{r}(Y, Z)$ and $0\leq B\in L^{r}(W, X)$. Since $M_{A,\,B}\geq0$,  we have $\|M_{A,\,B}\|_r=\|M_{A,\,B}\|\leq\|A\|\|B\|=\|A\|_r\|B\|_r$. On the other hand, for every $0\leq x^{\prime}\in X^{\prime}$ and every $0\leq y\in Y$ satisfying $\|x^{\prime}\|\leq1$ and $\|y\|\leq1$,  $x^{\prime}\otimes y\in L^{r}(X,\,Y)$ and $\|x^{\prime}\otimes y\|_r=\|x^{\prime}\otimes y\|\leq1$. Then it follows that
\begin{eqnarray*}
\|M_{A,\,B}\|_r=\|M_{A,\,B}\|&\geq&\sup\Big(\|M_{A,\,B}(x^{\prime}\otimes y)\|:0\leq x^{\prime}\in B_{X^{\prime}}, 0\leq y\in B_Y\Big)\\
&=&\sup\Big(\|(B^{\prime}x^\prime)\otimes Ay\|:0\leq x^{\prime}\in B_{X^{\prime}}, 0\leq y\in B_Y\Big)\\
&=&\|A\|\|B\|=\|A\|_r\|B\|_r.
\end{eqnarray*}
This implies that $\|M_{A,\,B}\|_r=\|A\|_r\|B\|_r$ holds for all $ 0\leq A\in L^{r}(Y, Z)$ and $0\leq B\in L^{r}(W, X)$.
\par Now, for the general case let  $A\in L^{r}(Y, Z)$ and $B\in L^{r}(W, X)$ be arbitrary. Then by Corollary \ref{Corollary 2.2}   we have
\begin{eqnarray*}
\|M_{A,\,B}\|_r=\||M_{A,\,B}|\|=\|M_{|A|,\,|B|}\|=\|M_{|A|,\,|B|}\|_r=\|A\|_r\|B\|_r.
\end{eqnarray*}
\end{proof}

Wickstead \cite{Wickstead} establishes that even in the case of atomic Banach lattices with order continuous norms  the operator norm of two-sided multiplication operators need not be equivalent to the regular norm.
\section{A Counterexample}
\par Let $X$ and $Y$ be Riesz spaces with $Y$ Dedekind complete. The set of all $\sigma$-order continuous operators in $L^{r}(X, Y)$ will be denoted by $L^{r}_c(X, Y)$.  The disjoint complement $(L^{r}_c(X, Y))^{d}$ of $L^{r}_c(X, Y)$ is denoted by $L^{r}_s(X, Y)$. Every element of $L^{r}_s(X, Y)$ is called a singular operator. When $Y=\mathbb{R}$, we write $X^{\sim}$ and $X^{\sim}_{s}$ respectively instead of $L^{r}(X, \mathbb{R})$ and $L^{r}_s(X, \mathbb{R})$. The following example illustrates that $|M_{A_0, B}|=M_{A_0, |B|}$ does not necessarily hold for all $B\in L^{r}(W, X)$,  that is, the linear operator $M_{A_0,\, \cdot}:B\in L^{r}(W, X)\rightarrow L^{r}\big(L^{r}(X, Y), L^{r}(W, Z)\big)$ is not necessarily a Riesz homomorphism in general.
\begin{example}  Let $W=X=Y=Z=\ell_{\infty}$ and let $e$ denote the strong unit $(1, 1,\cdots)$ of $\ell_{\infty}$. Let $0\le f\in (\ell_{\infty})^{\sim}_{s}$ be a singular Riesz homomorphism with $f(e)=1$ (one can take, e.g., $f$ equal to a limit over  a free ultrafilter). Let $B\in L^{r}(\ell_{\infty})$ be the rank one operator $B=f \otimes e$. Then it is clear that $B\in L^{r}_s(\ell_{\infty})$ and $I\wedge B=0$, where $I$ is the identity operator on $\ell_{\infty}$ (and hence order continuous). We claim that $M_{I,\,I}\wedge M_{I,\,B}\neq M_{I,\,I\wedge B}=0$.  To this end,
let $0\le T\in L^{r}(\ell_{\infty})$. Then, by \cite[Theorem 1.21(2)]{AB} we have
\begin{eqnarray*}
  \Bigg\{\sum\limits_{i=1}^{n}(T_{i}\wedge T_{i}B):n\in\mathbb{N}, T_{i}\ge 0, \sum\limits_{i} T_{i}=T \Bigg\}\downarrow (M_{I,\,I}\wedge M_{I,\,B})(T).
\end{eqnarray*}
From this  and \cite[Theorem 1.51(2)]{AB} it follows that
$$
\begin{array}{lc}
(M_{I,\,I}\wedge M_{I,\,B})(T)(e)\\[12pt]
\qquad\quad=\inf\Bigg\{\sum\limits_{i=1}^{n}(T_{i}\wedge T_{i}B)(e): n\in\mathbb{N}, T_{i}\ge 0, \sum\limits_{i} T_{i}=T \Bigg\}\\[12pt]
\qquad\quad=\inf\Bigg\{\sum\limits_{i=1}^{n}(T_{i}\wedge (f\otimes T_{i}e))(e): n\in\mathbb{N}, T_{i}\ge 0, \sum\limits_{i} T_{i}=T\Bigg\}\\[12pt]
\qquad\quad=\inf\Bigg\{\sum\limits_{i=1}^{n}\inf\Bigg(\sum\limits_{j=1}^{m_i}T_{i}x^{(i)}_{j}\wedge f(x^{(i)}_{j})T_{i}e:x^{(i)}_{j}\wedge x^{(i)}_{k}=0, j\neq k,\sum\limits_{j=1}^{m_i}x^{(i)}_{j}=e\Bigg): \\[12pt] \qquad\qquad\qquad\qquad\qquad\qquad\qquad\qquad\qquad\qquad\qquad\quad n\in\mathbb{N}, T_{i}\ge 0, \sum\limits_{i} T_{i}=T\Bigg\}.
\end{array}
$$
Let us put
 $$
 \begin{array}{l}
 G^{\,\prime}=\Bigg\{\sum\limits_{i=1}^{n}\inf\Bigg(\sum\limits_{j=1}^{m_i}T_{i}x^{(i)}_{j}\wedge f(x^{(i)}_{j})T_{i}e:x^{(i)}_{j}\wedge x^{(i)}_{k}=0, j\neq k,\sum\limits_{j=1}^{m_i}x^{(i)}_{j}=e\Bigg): \\[12pt]
  \qquad\qquad\qquad\qquad\qquad\qquad\qquad\qquad\qquad\qquad\quad n\in\mathbb{N}, T_{i}\ge 0, \sum\limits_{i} T_{i}=T \Bigg\},
 \end{array}
 $$
 $$
 \begin{array}{l}
 G^{\,\prime\prime}=\Bigg\{\sum\limits_{i=1}^{n}\sum\limits_{j=1}^{m}T_{i}x_{j}\wedge f(x_{j})T_{i}e:m\in\mathbb{N}, x_{j}\wedge x_{k}=0, j\neq k,\sum\limits_{j=1}^{m}x_{j}=e, \\[12pt]
  \qquad\qquad\qquad\qquad\qquad\qquad\qquad\qquad\qquad\qquad n\in\mathbb{N}, T_{i}\ge 0, \sum\limits_{i} T_{i}=T \Bigg\}.
 \end{array}
 $$
 We claim that $\inf G^{\,\prime}=\inf G^{\,\prime\prime}$. Indeed, it is clear that  $\inf G^{\,\prime}\leq\inf G^{\,\prime\prime}$. For the reverse inequality,
let $(T_i)^{n}_{1}$ be a fixed positive partition of $T$ (i.e., $ T_{i}\ge 0$, $\sum_{i} T_{i}=T$). For each $i$
let $(x^{(i)}_j)_{j=1}^{m_i}$ be an arbitrary positive disjoint partition of $e$ (i.e., $x^{(i)}_{j}\wedge x^{(i)}_{k}=0, j\neq k,\sum_{j=1}^{m_i}x^{(i)}_{j}=e$) corresponding to $T_i$. Following the proof of \cite[Theorem 1.51]{AB} we can find a  positive disjoint partition $(x_j)^{m}_1$ of $e$ such that
$$\sum\limits_{j=1}^{m}T_{i}x_{j}\wedge f(x_{j})T_{i}e\leq\sum\limits_{j=1}^{m_i}T_{i}x^{(i)}_{j}\wedge f(x^{(i)}_{j})T_{i}e  \qquad (i=1, 2,\cdot\cdot\cdot, n).$$
From this it follows that $$\inf G^{\,\prime\prime}\leq \sum\limits_{i=1}^{n}\sum\limits_{j=1}^{m}T_{i}x_{j}\wedge f(x_{j})T_{i}e\leq\sum\limits_{i=1}^{n}\sum\limits_{j=1}^{m_i}T_{i}x^{(i)}_{j}\wedge f(x^{(i)}_{j})T_{i}e $$Therefore, $$\inf G^{\,\prime\prime}\leq \sum\limits_{i=1}^{n}\inf\Bigg(\sum\limits_{j=1}^{m_i}T_{i}x^{(i)}_{j}\wedge f(x^{(i)}_{j})T_{i}e:x^{(i)}_{j}\wedge x^{(i)}_{k}=0, j\neq k,\sum\limits_{j=1}^{m_i}x^{(i)}_{j}=e\Bigg), $$which implies that $\inf G^{\,\prime\prime}\leq\inf G^{\,\prime}$.
Hence, we have $(M_{I,\,I}\wedge M_{I,\,B})(T)(e)=\inf G^{\,\prime\prime}$.
\par Since $f$ is a Riesz homomorphism, for every positive disjoint partition $(x_j)^{m}_1$ of $e$ appearing in $G^{\,\prime\prime}$ there exists only one, say $x_{j_0}$, such that $$f(x_j)=0, \quad j\neq j_0, \quad f(x_{j_0})=f(e)=1$$
 $$\quad\quad x_{j_0}\wedge\sum_{j\neq j_{0}}x_j=x_{j_0}\wedge(e-x_{j_0})=0.$$
It follows that $\sum_{i=1}^{n}\sum_{j=1}^{m}T_{i}x_{j}\wedge f(x_{j})T_{i}e= \sum_{i=1}^{n}T_{i}x_{j_0}.$ On the other hand, for any $x\in E^+$ satisfying $x\wedge(e-x)=0$ and $f(x)=1$, we must have $f(e-x)=0$, and hence $$\sum\limits_{i=1}^{n}T_{i}x=\sum_{i=1}^{n}\Big(T_{i}x\wedge f(x)T_{i}e+T_{i}(e-x)\wedge f(e-x)T_{i}e\Big)$$

Thus, we have $$\begin{array}{l}
(M_{I,\,I}\wedge M_{I,\,B})(T)(e)\\[10pt]
\qquad\qquad=\inf\Bigg\{\sum\limits_{i=1}^{n}\sum\limits_{j=1}^{m}T_{i}x_{j}\wedge f(x_{j})T_{i}e:m\in\mathbb{N}, x_{j}\wedge x_{k}=0, j\neq k,\sum\limits_{j=1}^{m}x_{j}=e, \\[12pt]
  \qquad\qquad\qquad\qquad\qquad\qquad\qquad\qquad\qquad\qquad\qquad\quad n\in\mathbb{N}, T_{i}\ge 0, \sum\limits_{i} T_{i}=T \Bigg\}\\[12pt]
\qquad\qquad=\inf\Bigg\{\sum\limits_{i=1}^{n}T_{i}x:x\wedge (e-x)=0, f(x)=1, n\in\mathbb{N}, T_{i}\ge 0, \sum\limits_{i} T_{i}=T\Bigg\}\\[12pt]
\qquad\qquad=\inf\Bigg\{Tx:0\le x\le e, x\wedge(e-x)=0, f(x)=1\Bigg\}.
\end{array}$$
If we now  take $T=B=f\otimes e$,  then $(M_{I,\,I}\wedge M_{I,\,B})(f\otimes e)(e)=e$. So, $M_{I,\,I}\wedge M_{I,\,B}\neq 0$.
\end{example}

\end{document}